\documentclass[11pt]{amsart}
\usepackage{amsmath}
\usepackage{}
\usepackage{amssymb,latexsym}
\usepackage{enumerate}

\newcommand{\op} {\overline{\partial}}
\newcommand{\dbar}{\ensuremath{\overline\partial}}

\newcommand{\C}{\ensuremath{\mathbb{C}}}
\newcommand{\R}{\ensuremath{\mathbb{R}}}

\newcommand{\B}{\ensuremath{\mathbb{B}}}

\makeatletter
\newcommand{\sumprime}{\if@display\sideset{}{'}\sum%
            \else\sum'\fi}
\makeatother

\begin{document}

\numberwithin{equation}{section}

\newtheorem{theorem}{Theorem}[section]
\newtheorem{proposition}[theorem]{Proposition}
\newtheorem{conjecture}[theorem]{Conjecture}
\def\theconjecture{\unskip}
\newtheorem{corollary}[theorem]{Corollary}
\newtheorem{lemma}[theorem]{Lemma}
\newtheorem{observation}[theorem]{Observation}
\newtheorem{definition}{Definition}
\numberwithin{definition}{section} 
\newtheorem{remark}{Remark}
\def\theremark{\unskip}
\newtheorem{kl}{Key Lemma}
\def\thekl{\unskip}
\newtheorem{question}{Question}
\def\thequestion{\unskip}
\newtheorem{example}{Example}
\def\theexample{\unskip}
\newtheorem{problem}{Problem}

\thanks{Research supported by the Key Program of NSFC No. 11031008.}

\address{Department of Mathematics, Tongji University, Shanghai, 200092, China}
\email{1113xuwang@tongji.edu.cn}

\title[Variation of Bergman kernels]{Variation of Bergman kernels of pseudoconvex domains}
 \author{Xu Wang}
\date{}
\maketitle

\begin{abstract} We shall give a variational formula of the full Bergman kernels associated to a family of smoothly bounded strongly pseudoconvex domains. An equivalent criterion for the triviality of holomorphic motions of planar domains in terms of the Bergman kernel is given as an application.
\bigskip


\smallskip

\end{abstract}

\tableofcontents

\section{Introduction}

This paper is an attempt to generalize Berndtsson's result on the plurisubharmonicity property of the Bergman kernel. Let $\phi$ be a plurisubharmonic function on a pseudoconvex domain $D$ in $\C^m_t \times \C^n_\zeta$. Denote by $K^t(\zeta,\bar \eta)$ the associated weighted (full) Bergman kernel on fibre $D_t$. Berndtsson (see Theorem 1.1 in \cite{Bern06}) proved that $\log K^t(\zeta,\bar \zeta)$ is plurisubharmonic on $D$ (see also \cite{Maitani84} and \cite{MY04} for eary results in this direction). The most important ingredient in his proof is a curvature property (denote by \textbf{C}) on product domain (see \cite{Bern09},\cite{Bern09a},\cite{Bern13} and \cite{BernPaun08} for applications). Our start point is to translate \textbf{C} to plurisubharmonicity property of the Bergman projection. Since Bergman projection has extremal property, the approximation technique in \cite{Bern06} (i.e., from product case to general case) still applies. Thus we get plurisubharmonicity property of the Bergman projection for general $D$ and $\phi$. Then plurisubharmonicity property of the Bergman kernel can be seen as as a special case (i.e., one point evaluation). By virtue of this observation, it is more natural to study variation of full Bergman kernels $K^t(\zeta,\bar \eta)$ (not only Bergman kernels on the diagonal), which is the theme of this paper. In order to get the second order variation formula for the Bergman kernel, it is necessary to find a $(1,0)$- vector field such that the associated representative of the Kodaira-Spencer class is primitive with respect to some complete K\"ahler metric. We shall show that how to use quasi-K\"ahler-Einstein metric to construct such vector field.

It is a pleasure to thank Bo Berndtsson for inspiring discussions on curvature properties of direct image sheaves. I am also grateful to Bo-Yong Chen for introducing me this topic.

\section{Basic notions and results}

Denote by $\mathbb B$ the unit ball in $\C^m$. Let $\{D_t\}_{t\in\mathbb B}$ be a family of smoothly bounded strongly pseudoconvex domains in $\C^n$. Assume that
\begin{equation*}
D:=\{(t,\zeta)\in\C^{m+n}: \zeta\in D_t, \ t\in\mathbb B\}
\end{equation*}
possesses a smooth defining function $\rho$ such that $\rho^t:=\rho|_{D_t}\in C^{\infty}(\overline{D_t})$ is a strictly plurisubharmonic defining function of $D_t$ for every $t$ in $D$.

Let $\phi$ be a smooth real function such that $\phi$ is smooth on a neighborhood of $\bar D\cap(\mathbb B\times \C^n)$. Put
\begin{equation*}
    \mathcal H_t:=\{f\in H^0(D_t,\wedge^nT^*D_t):~ i^{n^2}\int_{D_t}f\wedge\bar fe^{-\phi}<\infty\}.
\end{equation*}
Denote by $K^t(\zeta,\bar\eta)d\zeta\wedge\overline{d\eta}$ the Bergman kernel of $\mathcal H_t$, where $d\zeta$ is short for $d\zeta^1\wedge\cdots\wedge d\zeta^n$. Using Hamilton's theory on families of non-coercive boundary value problems (see \cite{Hamilton77}, \cite{Hamilton79} or \cite{GreeneK82}), one may prove that:

\begin{lemma}\label{le:stability}
For every fixed $\eta\in D_{t_0}$, $t_0\in \mathbb B$, $K^t(\zeta,\bar\eta)$ is smooth (as a function of $t_j,\bar t_k, \zeta_\alpha$) on a neighborhood of $\bar D\cap(U_{t_0}\times \C^n)$, where $U_{t_0}$ is a neighborhood of $t_0$.
\end{lemma}

Put $\cdot\bar\eta=K^t(\mu,\bar\eta)d\mu$. The above lemma implies that
\begin{equation*}
\cdot\bar\eta_{\bar j}:=\frac{\partial}{\partial \bar t^j}K^t(\mu,\bar\eta)d\mu\in \mathcal H_t,
\end{equation*}
for every fixed $\eta\in D_t$. By reproducing property of the Bergman kernels, we have
\begin{equation}\label{eq:reproducing-first}
    K_{\bar j}(\zeta,\bar\eta)=\langle\langle \cdot\bar\eta_{\bar j}, \cdot\bar\zeta\rangle\rangle, \ \
    K_{j}(\zeta,\bar\eta)=\langle\langle \cdot\bar\eta, \cdot\bar\zeta_{\bar j}\rangle\rangle.
\end{equation}

There is another way to compute $K_{j}(\zeta,\bar\eta)$ (without using reproducing property), i.e., to compute
\begin{equation*}
    \frac{\partial}{\partial t^j}\int_{D_t}\{\cdot\bar\eta,\cdot\bar\eta\},
\end{equation*}
where $\{\cdot,\cdot\}:=\cdot\wedge\bar\cdot e^{\phi}$ is the canonical sesquilinear pairing. Denote by $P$ the canonical projection from $\mathbb B\times\C^n$ to $\mathbb B$. Put
\begin{equation*}
    \mathcal V_j=\{V_j\in T(\rho) :~ P_* V_j=\partial/\partial t^j\},
\end{equation*}
where $T(\rho)$ is the space of all $(1,0)$-vector field $V$ such that $V$ is smooth on neighborhood of $\bar D\cap(\mathbb B\times \C^n)$ and $V(\rho)$ vanishes on $\partial D\cap(\mathbb B\times \C^n)$. Denote by $i_t$ the inclusion mapping $D_t\hookrightarrow D$. Since vector fields in $T(\rho)$ are tangent to the boundary, we can prove that
\begin{equation}\label{eq:Lie-first}
    \frac{\partial}{\partial t^j}\int_{D_t}\cdot=\int_{D_t}L_{V_j}\cdot=\int_{D_t}L^t_{V_j}\cdot,\ \ \forall \ V_j\in \mathcal V_j.
\end{equation}
where $L_{V_j}$ is the usual Lie-derivative and $L^t_{V_j}:=i_t^*L_{V_j}$. Thus we have
\begin{equation*}
    K_{j}(\zeta,\bar\eta)=i^{n^2}\int_{D_t} L_{V_j}\{\cdot\bar\eta,\cdot\bar\zeta\}.
\end{equation*}
By Cartan's formula
\begin{equation*}
    L_{V_j}=d\delta_{V_j}+\delta_{V_j}d, \ \delta_{V_j} \ \text{means contraction of a form with} \ V_j,
\end{equation*}
thus we have
\begin{equation*}
    K_{j}(\zeta,\bar\eta)=i^{n^2}\int_{\partial D_t} \delta_{V_j}\{\cdot\bar\eta,\cdot\bar\zeta\}+
    i^{n^2}\int_{D_t}  \frac{\partial}{\partial t^j}\{\cdot\bar\eta,\cdot\bar\zeta\}.
\end{equation*}
By \eqref{eq:reproducing-first}, the second term of the right hand side of the above equality is
\begin{equation*}
     2K_{j}(\zeta,\bar\eta)-i^{n^2}\int_{D_t}  \phi_j\{\cdot\bar\eta,\cdot\bar\zeta\},
\end{equation*}
thus we get Hadamard's first variational formula
\begin{equation}\label{eq:variation-first}
    K_{j}(\zeta,\bar\eta)=i^{n^2}\int_{D_t}  \phi_j\{\cdot\bar\eta,\cdot\bar\zeta\}
    -i^{n^2}\int_{\partial D_t} \delta_{V_j}\{\cdot\bar\eta,\cdot\bar\zeta\},
\end{equation}
which is due to Komatsu \cite{Komatsu82} (in case $\phi\equiv0$).

Notice that
\begin{equation*}
    \mathcal L_{V_j}^t\{\cdot,\cdot\}=\{L_j\cdot,\cdot\}+\{\cdot,L_{\bar j}\cdot\},
\end{equation*}
where $L_j$ (resp. $L_{\bar j}$) are short for $e^\phi L_{V_j}^t(e^{-\phi})$ (resp. $L_{\bar V_j}^t$). Since $L_{\bar j}\cdot\bar\eta=\cdot\bar\eta_{\bar j}$, \eqref{eq:reproducing-first} implies that $L_j\cdot\bar\eta ~\bot~\mathcal H_t$. Thus we have
\begin{equation*}
    0=\langle\langle L_j\cdot\bar\eta, \cdot\bar\zeta \rangle\rangle_{\bar k}=i^{n^2}\int_{D_t}\{L_j\cdot\bar\eta,L_k\cdot\bar\zeta\}+\{L_{\bar k}L_j\cdot\bar\eta,\cdot\bar\zeta\},
\end{equation*}
which implies that
\begin{equation}\label{eq:variation-second}
    K_{j\bar k}(\zeta,\bar\eta)=\langle\langle \cdot\bar\eta_{\bar k}, \cdot\bar\zeta_{\bar j} \rangle\rangle
    +\langle\langle [L_j,L_{\bar k}]\cdot\bar\eta,\cdot\bar\zeta\rangle\rangle
    -i^{n^2}\int_{D_t}\{L_j\cdot\bar\eta,L_k\cdot\bar\zeta\}.
\end{equation}
Denote by $_j\bar\eta^{n,0}$ (resp. $_j\bar\eta^{n-1,1}$) the $(n,0)$ (resp. $(n-1,1)$) -part of $L_j\cdot\bar\eta$. By Cartan's formula, we have
\begin{equation*}
    _j\bar\eta^{n,0}=i_t^*e^{\phi}(\partial\delta_{V_j}+\delta_{V_j}\partial)(e^{-\phi}\cdot\bar\eta),
    \ \ _j\bar\eta^{n-1,1}=\delta_{\dbar^tV_j}\cdot\bar\eta.
\end{equation*}
In deformation theory (see \cite{KodairaSpencer58}), $\dbar^tV_j$ represent Kodaira-Spencer classes.

Now the third term of the right hand side of \eqref{eq:variation-second} can be written as a Hermitian form
\begin{equation*}
    I_3=-\langle\langle {_j\bar\eta}^{n,0}, {_k\bar\zeta}^{n,0}\rangle\rangle
    -i^{n^2}\int_{D_t}\{{_j\bar\eta}^{n-1,1}, {_k\bar\zeta}^{n-1,1}\}.
\end{equation*}
Since $_j\bar\eta^{n,0}\bot\mathcal H_t$, $_j\bar\eta^{n,0}$ is the $L^2-$minimal solution of
\begin{equation*}
    \dbar^t(\cdot)=\dbar^t({_j\bar\eta^{n,0}}).
\end{equation*}
Notice that
\begin{equation}\label{eq:n0-n11}
    -\dbar^t({_j\bar\eta^{n,0}})=\partial_{\phi}^t({_j\bar\eta^{n-1,1}})+(\dbar^t\phi)_{V_j}\wedge\cdot\bar\eta,
\end{equation}
where $\partial^t_{\phi}:=e^{\phi}\partial^t(e^{-\phi})$ and $(\dbar^t\phi)_{V_j}=\sum (V_j\phi_{\bar\alpha})d\bar\mu^\alpha$. We shall use $L^2$-estimates to decode the positivity of $I_3$. In this way, we need that ${_j\bar\eta^{n-1,1}}$ is \textbf{primitive with respect to some complete K\"ahler metric} on $D_t$.

Primitivity is natural in the sense that
\begin{equation*}
    -i^{n^2}\int_{D_t}\{\cdot,\cdot\}=\langle\langle\cdot,\cdot\rangle\rangle
\end{equation*}
for primitive $(n-1,1)$-form. But if the metric is not complete, by H\"ormander's density lemma (see \cite{Hormander63} and \cite{Hormander65}), we have to estimate
\begin{equation*}
    \langle\langle\partial_{\phi}^t({_j\bar\eta^{n-1,1}}),u\rangle\rangle, \ \ \forall \ u\in {\rm Dom}(\dbar^t)^*\cap\ker\dbar^t\cap C^{\infty}(\overline{D_t}).
\end{equation*}
In order to use
\begin{equation}\label{eq:obstruction}
    \langle\langle\partial_{\phi}^t({_j\bar\eta^{n-1,1}}),u\rangle\rangle=\langle\langle {_j\bar\eta^{n-1,1}}, (\partial^t_{\phi})^*u\rangle\rangle,
\end{equation}
we have to assume that $u\in {\rm Dom}(\partial_{\phi}^t)^*$, which may not be true for $n\geq2$.

Notice that ${_j\bar\eta^{n-1,1}}$ is primitive with respect to K\"ahler form $\omega^t$ if
\begin{equation}\label{eq:primitive-vector}
    \dbar^t(\delta_{V_j}\omega^t)=0.
\end{equation}
Since $\omega^t$ is a K\"ahler metric. \eqref{eq:primitive-vector} is equivalent to the primitivity of the Kodaira-Spencer class. Now it suffices to find $V_j\in\mathcal V_j$ such that $\delta_{V_j}\omega^t$ is $\dbar^t$-closed for some complete
K\"ahler metric $\omega^t$. In general, if there exists $\psi$ such that $i\partial^t\dbar^t\psi=\omega^t$, then the unique solution $V:=V_j^{\psi}$ of
\begin{equation*}
    \delta_{V}\omega^t=-i\dbar^t\psi_j, \ P_*V=\partial/\partial t^j
\end{equation*}
of course satisfies \eqref{eq:primitive-vector}. Now we only need to \textbf{find $\psi$ such that $\omega^t$ is complete and $V_j^{\psi}$ is tangent to the boundary} (i.e., $V_j^{\psi}\in \mathcal V_j$).

By definition, $\{D_t\}$ is a smooth family of pseudoconvex domains with strictly plurisubharmonic definition functions $\rho|_{D_t}$. Thus
\begin{equation*}
    \omega^t:=i\partial^t\dbar^t(-\log-\rho).
\end{equation*}
is a complete K\"ahler metric on $D_t$. We call $\omega^t$ quasi-K\"ahler-Einstein metric on $D_t$. In fact, by Cheng-Yau's result \cite{CY80}, one may choose defining function $\rho_{CY}$ such that $i\partial^t\dbar^t(-\log-\rho_{CY})$ is a complete-K\"ahler-Einstein metric on $D_t$. To our surprise, every quasi-K\"ahler-Einstein metric fits our needs:

\begin{kl} Every $V_j^{-\log-\rho}$ is tangent to the boundary (i.e., $V_j^{-\log-\rho}\in\mathcal V_j$).
\end{kl}

Let's go back to \eqref{eq:variation-second}. Using integration by parts, we shall prove that the second term of the right hand side of \eqref{eq:variation-second} contains a boundary term with density
\begin{equation}\label{eq:k}
    b_{j\bar k}(\rho):=\frac{\langle V_j^{-\log-\rho}, V_k^{-\log-\rho}\rangle_{i\partial\op\rho}}{|\partial\rho|}.
\end{equation}
Denote by $_j\bar\eta$ the $\square'$-harmonic part of $_j\bar\eta^{n-1,1}$, where $\square'$ denotes the $\partial^t_{\phi}$-Laplace. Using our Key Lemma and $L^2$-estimates on complete K\"ahler manifolds \cite{Demailly82}, we can prove that:

\begin{theorem}\label{th:Bergman0} If $\phi\equiv0$ on $D$ then
\begin{equation}\label{eq:Bergman0}
    K_{j\bar k}(\zeta,\bar\eta)
    =\langle\langle \cdot\bar\eta_{\bar k}, \cdot\bar\zeta_{\bar j}\rangle\rangle
    +\int_{\partial D_t} b_{j\bar k}(\rho) \langle \cdot\bar\eta, \cdot\bar\zeta\rangle d \sigma
 + \langle\langle {_j\bar\eta}, {_k\bar\zeta}\rangle\rangle,
\end{equation}
where $d \sigma$ is the surface measure on $\partial D_t$ and $\langle \cdot\bar\eta, \cdot\bar\zeta\rangle:=K^t(\mu,\bar\eta)\overline{K^t(\mu,\bar\zeta)}e^{-\phi^t(\mu)}$.
\end{theorem}

As a direct corollary, we have:

\begin{corollary}\label{co:nakano} If $\phi\equiv0$ and $D$ is pseudoconvex then $\{\mathcal H_t\}$ is Nakano semi-positive.
\end{corollary}

Before we prove it, it is necessary to give a definition of the notion of Nakano positivity for a family of Hilbert spaces with infinite rank.

In finite dimensional (compact fibres) case, one may use variation of Bergman kernels to give an equivalent definition of Nakano positivity for direct image bundle $\{\mathcal H_t\}$. Put
\begin{equation*}
    K_{j\bar kp\bar q}=K_{j\bar k}(\bar\eta_q,\bar\eta_p)-\langle\langle \cdot\bar\eta_{p,\bar k},\cdot\bar\eta_{q, \bar j}\rangle\rangle,
\end{equation*}
where $\eta_p, \eta_q$ are points in $D_t, \ 1\leq p,q\leq r$, $r\in\mathbb Z_+$. We shall prove that (see Lemma ~\ref{le:Nakano-Bergman}) Nakano semi-positivity of $\{\mathcal H_t\}$ is equivalent to
\begin{equation}\label{eq:Bergman-Nakano}
    \sum c^{jp}\overline{c^{kq}}K_{j\bar kp\bar q}\geq0,
\end{equation}
where $r=\dim_{\C}\mathcal H_t$. In case $\dim_{\C}\mathcal H_t=\infty$, we say that $\{\mathcal H_t\}$ is Nakano semi-positive if the corresponding Bergman kernels satisfies \eqref{eq:Bergman-Nakano} for all $r\in\mathbb Z_+$.

By a similar argument as in the proof of Theorem~\ref{th:Bergman0}, we can also get a variational formula for weighted Bergman kernel. But it turns out to be an inequality instead of an equality. In case the weight function is plurisubharmonic, we shall prove that the corresponding direct image sheaf is Nakano semi-positive which can be seen as a generalization of Berndtsson's result (Theorem 1.1 in \cite{Bern09}) to non-product case (see also \cite{Tsuji05},\cite{Sch12}, \cite{Raufi13} and \cite{LiuYang13} for other results in this direction).

Assume that our weight function $\phi$ is smooth up to the boundary of $D$. Assume further that $\phi$ is strictly  plurisubharmonic along the fibres. Thus geodesic curvature of $\{\phi^t\}$
\begin{equation}\label{eq:geodesic curvature}
    c_{j\bar k}(\phi):=\phi_{j\bar k}-\sum\phi_{j\bar \alpha}\phi^{\bar \alpha\beta}\phi_{\bar k \beta}.
\end{equation}
is well defined. By Berndtsson's formula (see Lemma 4.1 in \cite{Bern11}),
\begin{equation*}
    \delta_{V_j^{\phi}}(\partial\dbar\phi)=\sum c_{j\bar k}(\phi)d\bar t^k.
\end{equation*}
We show prove that:

\begin{theorem}\label{th:Bergman1}  If $D$ is pseudoconvex and $\phi$ is plurisubharmonic on $D$ then $\{\mathcal H_t\}$ is Nakano semi-positive.
\end{theorem}

Berndtsson has informed the author that it would be also possible to define the notion of curvature $\Theta_{j\bar k}$ (as a densely defined closed operator) for $\{\mathcal H_t\}$. Then the proof of the above theorem implies that $\Theta_{j\bar k}-c_{j\bar k}(\phi)$ is Nakano semi-positive if $D$ is pseudoconvex.

Recall that our start point is to translate the curvature property to plurisubharmonicity property of the Bergman projection. Let $f$ be a smooth function with compact support in $D$ such that $f$ is a holomorphic function of $t$. Put
\begin{equation}\label{eq:kf}
    K_f(t)=\int_{D_t\times D_t}K^t(z,\bar w)f(t,z)\overline{f(t,w)}.
\end{equation}
Thus $K_f(t)$ is the square norm of the Bergman projection of $\bar fe^{\phi^t}$. Let $u^t$ be the $L^2$-minimal solution of $\op^t(\cdot)=\op^t(\bar fe^{\phi^t})$. Then $K_f(t)=||\bar fe^{\phi^t}||^2-||u^t||^2$. In compact case, Berndtsson and P\u{a}un showed that plurisubharmonicity of $K_f$ is equivalent to Griffiths positivity of the direct image bundle (see Proposition 3.4 in \cite{BernPaun08}). We shall prove that:

\begin{theorem}\label{th:Bern} The function $\log K_f(t)$ is plurisubharmonic on $\mathbb B$.
\end{theorem}

\begin{remark}\label{re:Bern} Using approximation technique in \cite{Bern06}, we shall prove that the above theorem is true for general pseudoconvex domain $D$ and plurisubharmonic function $\phi$ (i.e., no restriction on regularity of $D$ and $\varphi$).
\end{remark}

We want to point out that Theorem 1.3 in \cite{Wang13} is not true in general (i.e., the support of $f$ is necessary to be relatively compact in $D$). In fact, take $\phi=0$, if the above theorem is true for $f=1$ then $\log|D_t|$ is plurisubharmonic. However, Berndtsson showed that if $D_t$ is invariant under rotations $\cdot\mapsto e^{i\theta}\cdot$ then $-\log|D_t|$ is plurisubharmonic (see Theorem 1.2 in \cite{BernPaun08}).

Let's go back to \eqref{eq:variation-second} again. In one-dimensional case, \eqref{eq:obstruction} is always true. What's more, every form is primitive. Thus every vector field in $\mathcal V_j$ can be used to compute the variation.
In compact case (i.e., deformation of compact Riemann surfaces), Berndtsson (see \cite{Bern11}) showed that even if the curvature of the $0-$th direct image of the relative canonical line bundle vanishes identically, the Kodaira-Spencer class still happened to be non zero (i.e., the deformation is not trivial). Inspired by Berndtsson's idea, we shall use curvature of $\{\mathcal H_t\}$ to study triviality of holomorphic motions.

A homeomorphism $F:(t,z)\mapsto(t,f(t,z))$ from $\mathbb B\times D_0$ to $D$ is called a holomorphic motion (see \cite{MSS83}) of $D_0$ (with graph $D$) if $f(0,\cdot)$ is the identity mapping and $f(\cdot,z)$ is holomorphic for every fixed $z\in D_0$. $F$ is said to be a trivial motion of $D_0$ if there exists a bi-holomorphic mapping $G$ from $\mathbb B\times D_0$ to the graph of $F$ such that $G(\{t\}\times D_0)=F(\{t\}\times D_0)$ for every $t\in \mathbb B$ (i.e., there exists a fibre-preserving bi-holomorphic mapping from $\mathbb B\times D_0$ to $D$).

Consider the classical (i.e., $\phi\equiv0$) Bergman space $\mathcal H_t$ of the fibre $D_t:=F(\{t\}\times D_0)$. If  $D_0$ is a planar domain then the complex structure on each fibre can be represented by $J=f_{\bar z}/f_z$. We shall use variation of the Bergman kernels $K^t$ (or the curvature $\Theta_{j\bar k}$ of $\{\mathcal H_t\}$) to decode triviality of $F$.

\begin{theorem}\label{th:Cflat} Let $D_0$ be a smoothly bounded planar domain. Let $F$ be a holomorphic motion of $D_0$. If $F$ is smooth up to the boundary then the followings are equivalent:
\begin{enumerate}[\upshape (i)]
  \item $F$ is trivial.
  \item $\Theta_{j\bar k}\equiv0$, (i.e., $\sum c^{jp}\overline{c^{kq}}K_{j\bar kp\bar q}\equiv0$).
  \item For every $(t,\eta)$ in $D$ and every $j$,
  \begin{equation}\label{eq:trivialcondition}
    \int_{D_t} K^t(\zeta,\bar\eta) \left(\frac{(f_z)^2J_j}{|f_z|^2(1-|J|^2)}\right)(t,\zeta)\ id\zeta\wedge d\bar\zeta=0.
  \end{equation}
\end{enumerate}
\end{theorem}

As a direct corollary, we have:

\begin{corollary}\label{co:last} Let $F:(t,z)\mapsto(t,z+a(t)\bar z)$ be a holomorphic motion of a smoothly bounded planar domain. Then $F$ is trivial if and only if $a\equiv0$ on $\mathbb B$.
\end{corollary}

In \cite{Liurenshan}, Ren-Shan Liu showed that if $f=z+t^2\bar z$, then $F(\mathbb D\times \mathbb D)$ is not biholomorphic equivalent to the bidisc, where $\mathbb D$ denotes the unit disc. Corollary~\ref{co:last} is interesting, since every holomorphic motion of a subset of $\C$ can be extended to the whole complex plane (see \cite{Sl91} and \cite{ST86}). It is also interesting to study high-dimensional generalizations of Theorem~\ref{th:Cflat}.

\section{Variation of fibre integrals}

Let $\B$ be the unit ball in $\R^m$. Let $\{D_t\}_{t\in\B}$ be a family of smoothly bounded domain in $\R^n$. $\{D_t\}_{t\in\B}$ is said to be a smooth family if
\begin{equation*}
    D:=\{(t,x)\in\R^{m+n}: x\in D_t, \ t\in\B\}
\end{equation*}
possesses a smooth defining function $\rho$ such that $\rho|_{D_t}$ is a smooth defining function of $D_t$ for every $t$ in $\B$. Put
\begin{equation}\label{eq:boundary}
    [D]:=\overline D\cap(\B\times\R^n), \ \delta D:=\partial D\cap(\B\times\R^n).
\end{equation}
Let $dx:=dx^1\wedge\cdots\wedge dx^n$ denotes the Euclidean volume form on $\R^n$. Fix a smooth function $f$ on a neighborhood of $[D]$, the fibre integrals
\begin{equation*}
    F(t):=\int_{D_t}f(t,x)dx
\end{equation*}
depend smoothly on $t\in \B$. We shall introduce a natural way to compute the derivatives of $F(t)$ (see \cite{Sch12} for related results). For very fixed
$j\in\{1,\cdots m\}$, let
\begin{equation*}
    V_j:=\frac{\partial}{\partial t^j}-\sum v^{\alpha}_j\frac{\partial}{\partial x^\alpha}
\end{equation*}
be a smooth vector field on a neighborhood of $[D]$. We shall prove that:

\begin{theorem}\label{th:vfi} Let $\{D_t\}_{t\in\B}$ be a smooth family of smoothly bounded domain in $\R^m$. Assume that $V_j(\rho)$ vanishes on $\delta D$, then we have
\begin{equation}\label{eq:vfi}
 \frac{\partial F}{\partial t^j}(t)=\int_{D_t}L^t_{V_j}\left(f(t,x)dx\right),
\end{equation}
for every $t$ in $\B$.
\end{theorem}

\begin{proof} Without lose of generality, we may assume that $t=0$ and $j=1$. Since $V_1(\rho)$ vanishes on $\delta D$, the motion
\begin{equation*}
\Phi:(-1,1)\times D_0\rightarrow\R^{m}
\end{equation*}
of $D_0$ associated to $V_1$ is compatible with $\{D_t\}$, i.e.
\begin{equation*}
\Phi(a\times D_0)=D_{a\nu}, \ \nu=(1,0,\cdots,0)\in\R^p,
\end{equation*}
for every $a\in(-1,1)$. Since for every fixed $a\in(-1,1)$,
\begin{equation*}
\Phi^a:x\mapsto \Phi(a,x)
\end{equation*}
is a $C^{\infty}$ isomorphism from $D_0$ to $D_{a\nu}$, we have
\begin{equation}\label{eq:derivative}
     \frac{\partial F}{\partial t^1}(0)=\lim_{0\neq a\to 0}\int_{D_0}\frac{f(a\nu,\Phi^a(x))d\Phi^a(x)-f(0,x)dx}{a}
\end{equation}
Since $V_1$ and $f$ are smooth up to the boundary, we have
\begin{equation}\label{eq:derivative1}
     \frac{\partial F}{\partial t^1}(0)=\int_{D_0}\lim_{0\neq a\to 0}\frac{f(a\nu,\Phi^a(x))d\Phi^a(x)-f(0,x)dx}{a}.
\end{equation}
By definition of Lie derivative,
\begin{equation}\label{eq:derivative2}
L_{V_1}\left(f(t,x)dx\right)(0,x)=\lim_{0\neq a\to 0}\frac{[(\Psi^a)^*(fdx)](0,x)-f(0,x)dx}{a},
\end{equation}
where
\begin{equation*}
    \Psi^a:(b\nu,\Phi^b(x))\mapsto(b\nu+a\nu,\Phi^{b+a}(x)), \ (b,x)\in (-1+|a|,1-|a|)\times D_0.
\end{equation*}
Since
\begin{equation*}
    i_0^*\left\{[(\Psi^a)^*(fdx)](0,x)-f(a\nu,\Phi^a(x))d\Phi^a(x)\right\}=0,
\end{equation*}
\eqref{eq:vfi} follows from \eqref{eq:derivative1} and \eqref{eq:derivative2}.
\end{proof}

If $m=2$, put
\begin{equation*}
\frac{\partial}{\partial t}:=\frac12\left(\frac{\partial}{\partial t^1}-i\frac{\partial}{\partial t^2}\right), \
\frac{\partial}{\partial \bar t}:=\frac12\left(\frac{\partial}{\partial t^1}+i\frac{\partial}{\partial t^2}\right).
\end{equation*}
Let
\begin{equation*}
V=\frac{\partial}{\partial t}-\sum v^\alpha\frac{\partial}{\partial x^\alpha}
\end{equation*}
be a smooth vector field on a neighborhood of $[D]$. If $V(\rho)$ vanishes on $\delta D$, then both $2{\rm Re}V$ and $-2{\rm Im}V$ satisfy the condition of Theorem~\ref{th:vfi}. Thus we have:

\begin{corollary}\label{co:vfi} Assume that $V(\rho)$ vanishes on $\delta D$,  we have
\begin{equation*}
 \frac{\partial F}{\partial t}(t)=\int_{D_t} L^t_{V}\left(f(t,x)dx\right), \
 \frac{\partial F}{\partial \bar t}(t)=\int_{D_t} L^t_{\overline V}\left(f(t,x)dx\right),
\end{equation*}
for every $t\in \mathbb B$.
\end{corollary}

By Cartan's formula, $L_{V_j}=d\delta_{V_j}+\delta_{V_j}d$, and Theorem~\ref{th:vfi}, we have
\begin{equation}\label{eq:1stvariational}
     \frac{\partial F}{\partial t^j}(t)=\int_{\partial D_t}f(t,x)\delta_{V_j}dx+\int_{D_t}\frac{\partial f}{\partial t^j}(t,x)dx.
\end{equation}
One may also use Theorem~\ref{th:vfi} to compute variation of $\int_{\partial D_t}$. Fix a smooth form
\begin{equation*}
    g=\sum g^\alpha(t,x)\widehat{dx^\alpha}
\end{equation*}
on a neighborhood of $[D]$, where $\widehat{dx^\alpha}$ satisfies $dx^\alpha\wedge \widehat{dx^\alpha}=dx$. The fibre integrals
\begin{equation*}
    G(t):=\int_{\partial D_t}g
\end{equation*}
depend smoothly on $t\in \mathbb B$. Theorem~\ref{th:vfi} implies that

\begin{corollary}\label{co:vfibdy} Assume that $V_j(\rho)$ vanishes on $\delta D$,  we have
\begin{equation*}
 \frac{\partial G}{\partial t^j}(t)=\int_{\partial D_t}\delta_{V_j}dg=\int_{\partial D_t} L_{V_j}^t g,
\end{equation*}
for every $t\in \mathbb B$.
\end{corollary}

\begin{proof} By Stokes formula and Theorem~\ref{th:vfi}, we have
\begin{equation*}
     \frac{\partial G}{\partial t^j}(t)=\int_{\partial D_t} L_{V_j}^t d^tg,
\end{equation*}
where $d^t$ is the restriction of $d$ to $D_t$. Since
\begin{equation*}
    i_t^*L_{V_j}d^tg=i_t^*L_{V_j}dg,
\end{equation*}
we have
\begin{equation*}
     \frac{\partial G}{\partial t^j}(t)=\int_{D_t}d\delta_{V_j}dg
     =\int_{\partial D_t}\delta_{V_j}dg=\int_{\partial D_t}L_{V_j}^t g.
\end{equation*}
The proof is complete.
\end{proof} 

\section{Variation of Bergman kernels}

In this section we shall prove our results on the Bergman kernel stated in section 2.

\subsection{Stability of Bergman kernels}

We shall give an informal proof of Lemma~\ref{le:stability} by using regularity properties of \textbf{full} $\dbar$-Neumann problem (see \eqref{eq:full-Neumann} below). By Lemma 2.1 in \cite{Bern06}, stability of Bergman kernels follows directly from stability of solutions $u^t$ of a family of $\dbar$-Neumann problems $\square^t(\cdot)=f^t$. However, in general, it is not easy to show that $u^t$ is stable, i.e., if we want to use
\begin{equation}\label{eq:stability}
    ||\square^t(u^t-u^s)||=||f^t-f^s-(\square^t-\square^s)u^s||
\end{equation}
to estimate $||u^t-u^s||$, we have to find a natural connection between the domain of $\square^t$ and the domain of $\square^s$ (i.e., $u^s$ may not be in the domain of $\square^t$), but then we go back to regularity properties of $\square_b$-equation (see \cite{Kuranishi}).

Hamilton \cite{Hamilton79} found a more natural way to study regularity properties of families of non-coercive boundary value problems (not only for $\dbar$-Neumann problem). For reader's convenience we give a sketch description of Hamilton's idea.

Instead of considering $\square^t$ (whose domain satisfies the $\dbar$-Neumann condition), Hamilton considered the full Laplace opeartor $\widetilde{\square^t}$ (whose domain contains all forms smooth up to the boundary). Let $u^t$ be a form smooth up to the boundary, in  general, the Sobolev norm of $\widetilde{\square^t}(u^t)$ could not control the Sobolev norm of $u^t$. In fact, $u^t$ has to be in the domain of $\square^t$ (see \cite{Folland-Kohn72}). Thus two more operators (sending forms on $\overline{D_t}$ to forms on the boundary of $D_t$) are used in Hamilton's paper, i.e., he considered the full $\dbar$-Neumann problem
\begin{equation}\label{eq:full-Neumann}
    \mathfrak{S}^t(\cdot):=\left(\widetilde{\square^t},(\dbar^t\rho)\vee,(\dbar^t\rho)\vee\dbar^t\right)(\cdot)=f^t,
\end{equation}
where $(\dbar^t\rho)\vee\cdot:=(\dbar^t\rho\wedge\cdot)^*\cdot$. Now the domain of $\mathfrak{S}^t$ is $\C^{\infty}_{\bullet,\bullet}(\overline{D_t})$ for each $t$. Using $C^{\infty}$ trivialization mapping $\mathbb B\times D_0\to D$, the domain of $\mathfrak{S}^t$ can be seen as a fixed space $\C^{\infty}_{\bullet,\bullet}(\overline{D_0})$. Thus \eqref{eq:stability} applies. The only thing left to do is to show that universal constant (i.e., independent of $t$) works in the basic estimates for $\mathfrak{S}^t$. It is one of main results in \cite{Hamilton79}. The interested reader is referred to that paper for further information and a clear proof.

\subsection{$L^2$-estimates for $\dbar a=\partial_{\phi}b+c$}

Let $X$ be an $n-$dimensional complex manifold with complete K\"ahler metric $\omega$. Let $\phi$ be a smooth plurisubharmonic function on $X$. Denote by $\square'$ (resp. $\square''$) the $\partial_\phi$-Laplace (resp. $\dbar$-Laplace). Let $a$ (resp. $b$) be a smooth $L^2$-integrable $(n,0)$ (resp. $(n-1,1)$) form on $X$ such that $\dbar a$ (resp. $\partial_\phi b$) is $L^2$ on $X$ (here $L^2$ means $L^2$-integrable with respect to $\omega$ and $e^{-\phi}$). Assume that $a\bot\ker\dbar$ and $b$ is a $\dbar$-closed primitive form. Put $c=\dbar a-\partial_{\phi}b$. Since $\partial^*_{\phi}=-*\dbar*$, we know that $b$ is $\partial^*_{\phi}$-closed. Thus $b$ has orthogonal decomposition
\begin{equation*}
    b=b_1+\partial^*_{\phi}b_2,
\end{equation*}
where $b_1$ is the $\square'$-harmonic part of $b$. We shall prove the following Lemma (due to Berndtsson \cite{Bern11}) for reader's convenience.

\begin{lemma}\label{le:L2-1} If $\phi$ and $c$ are zero on $X$ then $||b||^2=||a||^2+||b_1||^2$.
\end{lemma}

\begin{proof} Denote by $G$ the Green operator with respect to $\square''$. Since $a$ is $L^2$-minimal, we have
\begin{equation*}
    a=\dbar^*G\partial b.
\end{equation*}
Thus
\begin{equation*}
    ||a||^2=\langle\langle G\partial b,\partial b\rangle\rangle.
\end{equation*}
Since $\omega$ is K\"ahler and $\phi\equiv 0$, we have $\square''=\square'$. Thus $G\partial b=b_2$, which implies that $||a||^2=||\partial^*b_2||^2=||b||^2-|b_1||^2$. The proof is complete.
\end{proof}

We remark that by a similar argument, one may show that
\begin{equation}\label{eq:L2-1}
    \langle\langle b^1, b^2\rangle\rangle=\langle\langle a^1,a^2\rangle\rangle+\langle\langle b_1^1, b_1^2\rangle\rangle.
\end{equation}

If $c$ is zero and $\phi$ is not assumed to be zero. Using K\"ahler identity $\square''-\square'=[i\partial\dbar\phi,\Lambda]$, where $\Lambda$ denotes the adjoint of $\omega\wedge\cdot$, one can also get an equality similar as \eqref{eq:L2-1}. If $c$ is not zero, we failed to find an equality as \eqref{eq:L2-1}. However, by using H\"ormander's $L^2$-estimates, we get an inequality between $a$, $b$ and $c$.

\begin{lemma}\label{le:L2-2} If $i\partial\dbar\phi>0$ on $X$ then $||a||^2\leq||b||^2-||b_1||^2+||c||^2_{i\partial\dbar\phi}$.
\end{lemma}

\begin{proof} Put $u=\partial_\phi b+c$. By H\"ormander's theory, it suffices to estimate
\begin{equation*}
    \langle\langle f,u\rangle\rangle=\langle\langle f,\partial_\phi (b-b_1)+c\rangle\rangle,
\end{equation*}
where $f$ is an arbitrary smooth $(n,1)$-form with compact support. The right hand side of the above inequality is $\langle\langle\partial_\phi^*f, (b-b_1)\rangle\rangle+\langle\langle f,c\rangle\rangle$. Thus
\begin{equation*}
    |\langle\langle f,u\rangle\rangle|^2\leq\left(||\partial^*_{\phi}f||^2+
    \langle\langle[i\partial\dbar\phi,\Lambda]f,f\rangle\rangle\right)
(||b-b_1||^2+||c||^2_{i\partial\dbar\phi}).
\end{equation*}
Since $\square''=\square'+[i\partial\dbar\phi,\Lambda]$, the right hand side of the above inequality is equal to
\begin{equation*}
    (||\dbar^*f||^2+||\dbar f||^2)(||b-b_1||^2+||c||^2_{i\partial\dbar\phi}).
\end{equation*}
Since $\omega$ is complete, we have
\begin{equation*}
    |\langle\langle f,u\rangle\rangle|^2 \leq (||\dbar^*f||^2+||\dbar f||^2)(||b-b_1||^2+||c||^2_{i\partial\dbar\phi})
\end{equation*}
for every $f$ in the domain of $\dbar\oplus\dbar^*$. Since $a$ is the $L^2$-minimal solution of $\dbar(\cdot)=u$. The proof is complete.
\end{proof}

\subsection{Second order variational formulas}

We shall prove Theorem~\ref{th:Bergman0} in this section. In order to use $L^2$-estimates in the previous section, we have to choose a suitable vector field, i.e., to prove our Key Lemma.

\begin{proof}[Proof of Key Lemma] Put $\psi=-\log-\rho$. By definition,
\begin{equation*}
    V_j^\psi=\partial/\partial t^j-\sum v_j^\alpha\partial/\partial \mu^\alpha,
\end{equation*}
where
\begin{equation*}
    v_j^\alpha=\sum\psi_{j\bar\beta}\psi^{\bar\beta\alpha}.
\end{equation*}
If $n=1$, it is easy to check that
\begin{equation}\label{eq:keylemma}
    V_j^\psi:=\frac{\partial}{\partial t^j}-\frac{\rho_j\rho_{\bar\mu}-\rho\rho_{j\bar\mu}}{|\rho_\mu|^2-\rho\rho_{\mu\bar\mu}}
    \frac{\partial}{\partial \mu},
\end{equation}
thus $V_j^\psi$ satisfies our Key Lemma. If $n\geq2$, fix $x_0\in\partial D_0$. Choosing suitable local coordinates around $x_0$, we may assume that
\begin{equation*}
    \left(\rho_{\alpha\bar \beta}(x_0)\right)=I_n, \ \rho_\nu(x_0)=0, \ \forall \ \nu\geq2,
\end{equation*}
where $I_n$ is the identity matrix. Thus
\begin{equation*}
    v_j^1(x_0)=\frac{\rho_j\rho_{\bar1}-\rho\rho_{j\bar1}}{|\rho_1|^2-\rho}(x_0)=\frac{\rho_j}{\rho_1}(x_0), \  v_j^\alpha(x_0)=\rho_{j\bar \alpha}(x_0), \ \forall \ \alpha\geq2,
\end{equation*}
which implies that $V_j^{\psi}$ is smooth up to the boundary (one may also prove this by force). Now
\begin{equation*}
    V_j^{\psi}(\rho)(x_0)=\rho_j(x_0)-\sum v_j^\alpha\rho_\alpha(x_0)=\rho_j(x_0)-\rho_j(x_0)=0.
\end{equation*}
The proof of Key Lemma is complete.
\end{proof}

Now we can use the vector fields $V_j^{\psi}$ in our Key Lemma to compute variation of Bergman kernels. By \eqref{eq:n0-n11} and Lemma~\ref{le:L2-1}, the last term in \eqref{eq:variation-second} is equal to the last term in \eqref{eq:Bergman0}. Thus Theorem~\ref{th:Bergman0} follows from the following (integration by parts) lemma:

\begin{lemma}\label{le:bdy} If $\phi\equiv0$ then we have
\begin{equation}\label{eq:bdy1}
    \langle\langle [L_j,L_{\bar k}]\cdot\bar\eta,\cdot\bar\zeta\rangle\rangle
    =\int_{\partial D_t} b_{j\bar k}(\rho) \langle \cdot\bar\eta, \cdot\bar\zeta\rangle d \sigma.
\end{equation}
For general $\phi$ with well defined geodesic curvature $c_{j\bar k}(\phi)$, we have
\begin{equation}\label{eq:bdy2}
    \langle\langle [L_j,L_{\bar k}]\cdot\bar\eta,\cdot\bar\zeta\rangle\rangle
    =\int_{\partial D_t} b_{j\bar k}(\rho) \langle \cdot\bar\eta, \cdot\bar\zeta\rangle d \sigma+
    \langle\langle c_{j\bar k}(\phi,V)\cdot\bar\eta,\cdot\bar\zeta\rangle\rangle,
\end{equation}
where
\begin{equation*}
    c_{j\bar k}(\phi,V):=c_{j\bar k}(\phi)+\langle(\op^t\phi)_{V^{\psi}_j}, (\op^t\phi)_{V^{\psi}_k}\rangle_{i\partial^t\op^t\phi}.
\end{equation*}
\end{lemma}

\begin{proof} We shall only prove the following special case: $j=k=m=1$, since the general case follows by a similar argument.

Put $V=V_1^{\psi}$. Notice that $[L_1,L_{\bar 1}]=\overline VV\phi+[L^t_{V},L^t_{\overline{V}}]$. By Cartan's formula, we have $(L_V-L_V^t)\cdot\bar\eta=\delta_{d_tV}\cdot\bar\eta$, where $d_t:=\frac{\partial}{\partial t}\otimes dt+\frac{\partial}{\partial \bar t}\otimes d\bar t$. Thus
\begin{equation*}
    \left(L_{\overline V}^t L_V^t \cdot\bar\eta\right)_{(n,0)} =
    \left(L_{\overline V} L_V \cdot\bar\eta\right)_{(n,0)},
\end{equation*}
where $(\cdot)_{(n,0)}$ means the $(n,0)$-component of $i_t^*(\cdot)$. Since $L_{\overline V}\cdot\bar\eta=\cdot\bar\eta_{\bar t}=L_{\overline V}^t\cdot\bar\eta$, we have
\begin{equation*}
    \left([L_V^t,L_{\overline V}^t] \cdot\bar\eta\right)_{(n,0)} =
    \left([L_V,L_{\overline V}] \cdot\bar\eta\right)_{(n,0)}.
\end{equation*}
Since $[L_V,L_{\overline V}]=L_{[V,\overline V]}$, we have
\begin{equation*}
    \left([L_V^t,\mathcal L_{\overline V}^t] \cdot\bar\eta\right)_{(n,0)}
    =\partial^t\delta_{[V,\overline V]}\cdot\bar\eta.
\end{equation*}
Put $V=\frac{\partial}{\partial t}-\sum v^\alpha\frac{\partial}{\partial \mu^\alpha}$, we have
\begin{eqnarray*}
  \overline VV\phi-\langle(\op^t\phi)_V, (\op^t\phi)_V\rangle_{i\partial^t\op^t\phi} &=& c(\phi)-\sum\left(v^\alpha_{\bar t}\phi_\alpha-\overline{v^\beta}v^\alpha_{\bar\beta}\phi_\alpha\right) \\
   &=& c(\phi)-\delta_{[V,\overline V]}\partial^t\phi,
\end{eqnarray*}
where $c(\phi)$ is short for $c_{1\bar 1}(\phi)$. Thus
\begin{equation*}
    \left([L_1, L_{\bar 1}]\cdot\bar\eta\right)_{(n,0)}
    =c(\phi,V)\cdot\bar\eta+\partial_{\phi}^t\delta_{[V,\overline V]}\cdot\bar\eta,
\end{equation*}
where $c(\phi,V)$ is short for $c_{1\bar 1}(\phi,V)$. It suffices to show that
\begin{equation}\label{eq:bdy3}
    \langle\langle \partial_{\phi}^t\delta_{[V,\overline V]}\cdot\bar\eta, \cdot\bar\zeta \rangle\rangle
    =\int_{\partial D_t} b(\rho) \langle \cdot\bar\eta, \cdot\bar\zeta\rangle d \sigma,
\end{equation}
where $b(\rho)$ is short for $b_{1\bar1}(\rho)$. Notice that the left hand side of the above equality is equal to
\begin{equation*}
    i^{n^2}\int_{\partial D_t}\{\delta_{[V,\overline V]}\cdot\bar\eta, \cdot\bar\zeta\}
    =\int_{\partial D_t}\frac{\sum \left(V\overline{v^\alpha}\right)\rho_{\bar \alpha}}{|\partial \rho|}\langle \cdot\bar\eta,\cdot\bar\zeta\rangle d\sigma.
\end{equation*}
Since
\begin{equation*}
    \langle V, V\rangle_{i\partial\op\rho}-\sum \left(V\overline{v^\alpha}\right)\rho_{\bar\alpha}=V\overline V\rho
\end{equation*}
and $V\overline V\rho=0$ on the boundary of $D_t$, we get \eqref{eq:bdy3}. The proof is complete.
\end{proof}

We remark that the only property of the vector field $V_j^{\psi}$ used in the proof of the above lemma is $V_j^{\psi}\rho=0$ on the boundary. Thus the above lemma is true for every $V_j\in\mathcal V_j$.

Notice that if every fiber is one-dimensional then every vector field in $\mathcal V_j$ can be used to computed the variation. What's more,
\begin{equation*}
    \frac{\langle V_j, V_k\rangle_{i\partial\dbar\rho}}{|\partial\rho|} \equiv
    \frac{ \rho_{j\bar k} |\rho_\mu|^2 -\rho_{j\bar{\mu}} \rho_{\bar k} \rho_{\mu}
    -\rho_{\bar{k}\mu} \rho_{j} \rho_{\bar{\mu}} + \rho_{j} \rho_{\bar{k}} \rho_{\mu\bar{\mu}}}{|\rho_\mu|^3}
\end{equation*}
does not depend not $V_j\in \mathcal V_j$, $V_k\in \mathcal V_k$. Thus we have

\begin{theorem}\label{th:planar} Let $\{D_t\}_{t\in\mathbb B}$ be a smooth family of smoothly bounded planar domains. If $\phi\equiv0$ on $D$ then we have
\begin{equation*}
 K_{j\bar k}(\zeta,\bar\eta) = \int_{\partial D_t} b_{j\bar k}(\rho) \langle \cdot\bar\eta, \cdot\bar\zeta\rangle d \sigma +\langle\langle \cdot\bar\eta_{\bar k}, \cdot\bar\zeta_{\bar j}\rangle\rangle + \langle\langle {_j\bar\eta}, {_k\bar\zeta}\rangle\rangle
\end{equation*}
where ${_j\bar\eta}$ is the harmonic part of $\delta_{\op^t V_j}\cdot\bar\eta$.
\end{theorem}

We shall use the above theorem to study triviality of holomorphic motions.

\subsection{Bergman kernel and curvature property}

We shall prove Corollary~\ref{co:nakano}, Theorem~\ref{th:Bergman1} and Theorem~\ref{th:Bern} in this section.

Let $\pi:D\to \mathbb B$ be a proper holomorphic submersion. Let's recall the definition of Nakano positivity for holomorphic vector bundle $\{\mathcal H_t\}$ associated to $0$-th direct image $\pi_*(K_{D/\mathbb B}) $
(we assume that $\dim_\C\mathcal H_t$ is a constant $r$). Thus $\mathcal H_t$ can be seen as the Bergman space of the fibre at $t$. By definition, $\{\mathcal H_t\}$ is said to be Nakano semi-positive if for every $u^1,\cdots,u^m\in\mathcal H_t$,
\begin{equation}\label{eq:Def-Nakano}
    \sum\langle\langle\Theta_{j\bar k}u^j,u^k\rangle\rangle \geq 0,
\end{equation}
where $\Theta_{j\bar k}$ is the curvature of the Chern connection on $\{\mathcal H_t\}$. We may choose $\cdot\bar\eta_p$, $p=1,\cdots,r$, such that $\mathcal H_t={\rm Span}\{\cdot\bar\eta_p\}$. Thus every $u^j$ can be written as  $u^j=\sum c^{jp}\cdot\bar\eta_p$. Hence \eqref{eq:Def-Nakano} is equivalent to
\begin{equation}\label{eq:Def-Nakano-Bergman}
    \sum c^{jp}\overline{c^{kq}}\langle\langle\Theta_{j\bar k}\cdot\bar\eta_p,\cdot\bar\eta_q\rangle\rangle \geq 0.
\end{equation}
Denote by $D_j$ the contraction of $\partial/\partial t^j$ with $(1,0)$-component of the Chern connection on $\{\mathcal H_t\}$. By definition of the Chern connection, $D_j\cdot\bar\eta_p$ is the Bergman projection of $L_j\cdot\bar\eta_p$. By \eqref{eq:reproducing-first}, $L_j\cdot\bar\eta_p~\bot~\mathcal H_t$. Thus $D_j\cdot\bar\eta_p=0$. Since $\Theta_{j\bar k}=[D_j,\partial/\partial \bar t^k]$, we have
\begin{equation*}
    \langle\langle\Theta_{j\bar k}\cdot\bar\eta_p,\cdot\bar\eta_q\rangle\rangle=K_{j\bar k p\bar q}.
\end{equation*}
Thus we get the following lemma:

\begin{lemma}\label{le:Nakano-Bergman} Nakano semi-positivity of $\{\mathcal H_t\}$ is equivalent to \eqref{eq:Bergman-Nakano}.
\end{lemma}

In case $\dim_\C\mathcal H_t=\infty$, $\{\mathcal H_t\}$ is said to be Nakano semi-positive if \eqref{eq:Bergman-Nakano} is true for every positive integer $r$. Thus Corollary~\ref{co:nakano} is a direct corollary of Lemma~\ref{le:bdy} and Theorem~\ref{th:Bergman0}.

\begin{proof}[Proof of Theorem~\ref{th:Bergman1}]
Since $D$ is pseudoconvex and $\phi$ is plurisubharmonic, by Lemma~\ref{le:bdy}, we have
\begin{equation*}
    \sum c^{jp}\overline{c^{kq}}\langle\langle [L_j,L_{\bar k}]\cdot\bar\eta_p,\cdot\bar\eta_q\rangle\rangle
    \geq ||c||^2_{i\partial^t\dbar^t\phi},
\end{equation*}
where
\begin{equation*}
    c:=\sum c^{jp}(\dbar^t\phi)_{V_j^\psi}\wedge\cdot\bar\eta_p.
\end{equation*}
The last term in \eqref{eq:variation-second} can be written as $||b||^2-||a||^2$, where
\begin{equation*}
    b=\sum c^{jp}{_j\bar\eta_p^{n-1,1}}, \ \ a=-\sum c^{jp}{_j\bar\eta_p^{n,0}}.
\end{equation*}
Thus if suffices to show that $||b||^2-||a||^2+||c||^2_{i\partial^t\dbar^t\phi}\geq 0$. By \eqref{eq:n0-n11},
$\dbar^t a=\partial^t_{\phi}b+c$. Thus Theorem~\ref{th:Bergman1} follows from Lemma~\ref{le:L2-2}.
\end{proof}

\begin{proof}[Proof of Theorem~\ref{th:Bern}] We may assume $m=1$. Put
\begin{equation*}
    P_f(z)=\int_{D_t}K^t(z,\bar w)\overline{f(t,w)}.
\end{equation*}
We claim that $D_tP_f:=e^\phi\partial/\partial t(P_fe^{-\phi})$ is perpendicular to the Bergman space: It suffices to show that $\langle\langle h, D_t(P_f)\rangle\rangle=0$ for every function $h$ holomorphic on a neighborhood of the closure of $D_t$. Thus $h$ can be seen as a holomorphic function on nearby fibres and
\begin{equation*}
    0=\partial/\partial \bar t\int_{D_t}h f=\partial/\partial \bar t\int_{D_t} h \overline{P_f}e^{-\phi}=\langle\langle h, D_tP_f\rangle\rangle.
\end{equation*}
Our claim is proved.

Since $K_f(t)=||P_f||^2$, we have
\begin{equation*}
    K_{f,t\bar t}=||P_{f,\bar t}||^2+\langle\langle \phi_{t\bar t}P_f,P_f\rangle\rangle-||D_tP_f||^2.
\end{equation*}
Notice that $\dbar^t(D_tP_f)=-P_f\dbar^t(\phi_t)$. By Lemma~\ref{le:L2-2},
\begin{equation*}
    ||D_tP_f||^2\leq \langle\langle|\dbar^t\phi_t|^2_{i\partial^t\dbar^t\phi}P_f,P_f \rangle\rangle.
\end{equation*}
Thus
\begin{equation*}
    K_{f,t\bar t}\geq||P_{f,\bar t}||^2+\langle\langle c(\phi)P_f,P_f\rangle\rangle,
\end{equation*}
which implies that
\begin{equation*}
    (\log K_f)_{t\bar t}\geq\frac{\langle\langle c(\phi)P_f,P_f\rangle\rangle}{||P_f||^2}\geq 0.
\end{equation*}
The proof is complete.
\end{proof}

Now we shall use Berndtsson's approximation technique (see section 3 in \cite{Bern06}) to prove the remark behind Theorem~\ref{th:Bern}.

\begin{proof}[Proof of the the remark behind Theorem~\ref{th:Bern}] Notice that $K_f$ satisfies the foolowing extremal property:
\begin{equation*}
    K_f(t)=\sup_{h\in \mathcal H_t}\{|\int_{D_t}hf|^2/||h||^2\}.
\end{equation*}
Since the Bergman kernel associated to $\phi$ and $D$ is a decreasing limit of the Bergman kernel associated to smooth weight and smooth strictly pseudoconvex domain. We know that $\log K_f$ is plurisubharmonic on $\mathbb B$ for general $\phi$ and $D$.
\end{proof} 

\section{Applications to holomorphic motions}

We shall prove Theorem~\ref{th:Cflat} in this section. Let $D_0$ be a smoothly bounded planar domain. Let $F$ be a holomorphic motion of $D_0$. If $F$ is smooth up to the boundary, then the vector field $V^F_j$ defined by $V^F_j:=F_*\left(\frac{\partial}{\partial t^j}\right)$ is tangent to the boundary (i.e., $V^F_j\in\mathcal V_j$). Thus Theorem~\ref{th:planar} applies. Since the graph $D$ of $F$ is Levi-flat, we have
\begin{equation*}
\sum c^{jp}\overline{c^{kq}}\int_{\partial D_t} b_{j\bar k}(\rho) \langle \cdot\bar\eta_{p}, \cdot\bar\eta_q\rangle d \sigma=0,
\end{equation*}
which implies that
\begin{equation}\label{eq:holo-motion-cur}
\sum c^{jp}\overline{c^{kq}}K_{j\bar kp\bar q}=||\sum c^{jp}{_j\bar\eta_p}||^2.
\end{equation}

\begin{proof}[Proof of Theorem~\ref{th:Cflat}] (i) $\Rightarrow$ (ii): If $F$ is trivial then there exists a bi-holomorphic mapping $G$ with the same fibres. Thus $V^G_j$ are holomorphic. Hence ${_j\bar\eta_p}\equiv0$. Then \eqref{eq:holo-motion-cur} implies (ii).

(ii) $\Rightarrow$ (iii): If (ii) is true then ${_j\bar\eta}\equiv0$ by \eqref{eq:holo-motion-cur}. Thus
$\delta_{\op^t V^F_j}\cdot\bar\eta\perp\ker\partial^t$. Since $d\bar\zeta\in\ker\partial^t$, we have
\begin{equation}\label{eq:step2}
    \int_{D_t}f_{j\bar\zeta}K^t(\zeta,\bar\eta)\ id\zeta\wedge d\bar\zeta=0.
\end{equation}
Notice that
\begin{equation*}
    \overline{z_{\zeta}}=\frac{f_z}{|f_z|^2-|f_{\bar z}|^2}, \ z_{\bar\zeta}=\frac{-f_{\bar z}}{|f_z|^2-|f_{\bar z}|^2},
\end{equation*}
we have
\begin{equation}\label{eq:abz}
    f_{j\bar\zeta}=f_{jz}z_{\bar\zeta}+f_{j\bar z}\overline{z_{\zeta}}=\frac{(f_z)^2J_j}{|f_z|^2(1-|J|^2)}.
\end{equation}
Thus (iii) is true.

(iii) $\Rightarrow$ (i): It suffices to find holomorphic vector fields $V_j$, $j=1,\cdots,m$, on $D$ such that $V_j=V_j^F$ on $\partial D$ and $[V_j,V_k]=0$ on $D$ . By \eqref{eq:step2}, there exits $g^{t,j}$ such that
\begin{equation}\label{eq:step31}
    \op^t f_j=(\partial^t)^*(g^{t,j} id\zeta\wedge d\bar\zeta)=-i\op^tg^{t,j}.
\end{equation}
Take $g^j$ such that $g^j|_{D_t}=g^{t,j}$. We claim that
\begin{equation*}
    V_j=\partial/\partial t^j+(f_j+ig^j)\partial/\partial \zeta, \ \ j=1,\cdots,m,
\end{equation*}
fit our needs. Since $g^{t,j} id\zeta\wedge d\bar\zeta\in{\rm Dom}(\partial^t)^*$, we have $g^j=0$ on $\partial D$ (i.e., $V_j=V_j^F$ on $\partial D$). Thus it suffices to prove that $V_j$ are holomorphic and integrable.

By \eqref{eq:step31}, $f_{j\bar k}+ig^j_{\bar k}$ are holomorphic on each fibre. To prove $f_{j\bar k}+ig^j_{\bar k}=0$ on each fibre, it suffices to show that $f_{j\bar k}+ig^j_{\bar k}=0$ on the boundary of each fibre.

Since $g^j=0$ on $\partial D$, we have
\begin{equation}\label{eq:step32}
    V_k^F g^j=g^j_k+f_kg^j_\zeta=0, \ \ \text{on} \  \partial D,
\end{equation}
and
\begin{equation}\label{eq:step33}
    \overline{V_k^F} g^j=g^j_{\bar k}+\overline{f_k}g^j_{\bar \zeta}=0, \ \ \text{on} \  \partial D.
\end{equation}
By definition of $V_j^F$, we have $[V_j^F,V_k^F]=0$ and $[V_j^F,\overline{V_k^F}]=0$. Thus
\begin{equation}\label{eq:step34}
    f_jf_{k\zeta}-f_kf_{j\zeta}=0, \ \ f_{j\bar k}+\overline{f_k}f_{j\bar \zeta}=0, \ \ \text{on} \  D.
\end{equation}
By \eqref{eq:step31}, \eqref{eq:step33} and \eqref{eq:step34}, we have
\begin{equation*}
    0=\overline{f_k}f_{j\bar \zeta}+i\overline{f_k}g^j_{\bar \zeta}=- f_{j\bar k}-ig^j_{\bar k},\ \ \text{on} \  \partial D.
\end{equation*}
Thus $f_{j\bar k}+ig^j_{\bar k}\equiv0$ on $D$. By \eqref{eq:step31}, $V_j$ are holomorphic on $D$.

Now we need to show that $[V_j,V_k]=0$ on $D$. Since $[V_j,V_k]$ are holomorphic, it suffices to show that $[V_j,V_k]=0$ on $\partial D$. Since $g^j=0$ on $\partial D$, we have
\begin{equation*}
    [V_j,V_k]=f_jf_{k\zeta}-f_kf_{j\zeta}+i(f_jg^k_\zeta+g^k_j)-i(f_kg^j_\zeta+g^j_k), \ \ \text{on} \  \partial D.
\end{equation*}
By \eqref{eq:step32} and \eqref{eq:step34}, $[V_j,V_k]=0$ on $\partial D$. Thus $V_j$ are integrable and our claim is proved. The proof is complete.
\end{proof}

\begin{proof}[Proof of Corollary~\ref{co:last}] By definition of $F$, (iii) is equivalent to $a_j(t)\equiv0$. Since $a(0)=0$, (iii) is equivalent to $a\equiv0$. Thus Corollary~\ref{co:last} follows from Theorem~\ref{th:Cflat}.
\end{proof}

\end{document}